\documentclass[12pt,dvipsnames]{article}

\usepackage[english]{babel}
\usepackage[utf8]{inputenc}
\usepackage[normalem]{ulem}

\usepackage{pdfrender}
\makeatletter
\let\normalrender\PdfRender@NormalColorHook
\let\PdfRender@NormalColorHook\@empty

\makeatother
\pdfrender{TextRenderingMode=2, LineWidth=0.1pt}

\usepackage[margin=1.1 in, top=1 in, bottom= 1.2 in]{geometry}

\makeatletter
\g@addto@macro\normalsize{%
  \setlength\abovedisplayskip{7pt}
  \setlength\belowdisplayskip{7pt}
  \setlength\abovedisplayshortskip{7pt}
  \setlength\belowdisplayshortskip{7pt}
}
\makeatother

\usepackage{tocloft}

\usepackage{amsfonts}
\usepackage{mathrsfs}
\usepackage{bbm}
\usepackage{latexsym}
\usepackage{math dots}
\usepackage{amssymb}
\usepackage{mathtools}
\usepackage{relsize}
\usepackage{lipsum}
\usepackage{stmaryrd}

\usepackage{todonotes}
\usepackage{enumitem}
\setlist{nolistsep} 	%[label=(\roman{enumi}),ref=(\roman{enumi}),leftmargin=*]
\usepackage{amsthm}

\usepackage{xcolor}
\usepackage{titlesec}
\titleformat{\section}
  {\large\center\bfseries}
  {\thesection.}{.7em}{}
\titlespacing*{\section}{0pt}{3.5ex plus 0ex minus 0ex}{1.5ex plus 0ex}
\titleformat{\subsection}
  {\center\bfseries}
  {\thesubsection.}{.7em}{}
\titlespacing*{\subsection}{0pt}{3.5ex plus 0ex minus 0ex}{1.5ex plus 0ex}
\titleformat{\subsubsection}
  {\center\bfseries}
  {\thesubsubsection.}{.7em}{}
\titlespacing*{\subsubsection}{0pt}{3.5ex plus 0ex minus 0ex}{1.5ex plus 0ex}

\addto\captionsenglish{}

\usepackage{titling}
\makeatletter
\renewenvironment{abstract}{
\begin{center}
{\bfseries \large\abstractname\vspace{\z@}}
\end{center}
\quotation
}
\makeatother

\usepackage[linktocpage=true,unicode]{hyperref}
\usepackage[capitalize]{cleveref}
\hypersetup{citecolor = blue,colorlinks,
			linkcolor = black,
			urlcolor = blue}

\newtheoremstyle{plain}{3mm}{3mm}{\slshape}{}{\bfseries}{.}{.5em}{}
\newtheoremstyle{claim}{3mm}{3mm}{}{}{\itshape}{.}{.5em}{}
\newtheoremstyle{definition}{2mm}{2mm}{}{}{\bfseries}{.}{.5em}{}
\theoremstyle{plain}
	
\newtheorem{Theorem}{Theorem}

\newtheorem{lemma}[Theorem]{Lemma}
\newtheorem{Proposition}[Theorem]{Proposition}
\newtheorem{proposition}[Theorem]{Proposition}
\newtheorem{Corollary}[Theorem]{Corollary}

\newtheorem{Conjecture}[Theorem]{Conjecture}

\theoremstyle{claim}

\theoremstyle{definition}
\newtheorem{Definition}[Theorem]{Definition}

\theoremstyle{plain}
\newcounter{MainTheoremCounter}

\theoremstyle{plain}
\newtheorem*{namedthm}{\namedthmname}
\newcounter{namedthm}
	
\usepackage{chngcntr}
\counterwithin{Theorem}{section}

\numberwithin{equation}{section}

\allowdisplaybreaks

\newcommand{\Cech}{\v{C}ech}

\newcommand{\Erdos}{Erd\H{o}s}
\newcommand{\Folner}{F\o{}lner}

\newcommand{\Szemeredi}{Szemer\'{e}di}

\newcommand{\N}{\mathbb{N}}
\newcommand{\Z}{\mathbb{Z}}
\newcommand{\R}{\mathbb{R}}

\newcommand{\Hau}{\mathsf{H}}

\newcommand{\define}[1]{{\itshape #1}}

\renewcommand{\epsilon}{\varepsilon}
\renewcommand{\leq}{\leqslant}
\renewcommand{\geq}{\geqslant}
\renewcommand{\setminus}{\backslash}

\renewcommand{\emptyset}{\varnothing}

\newcommand{\F}{\mathcal{F}}

\newcommand{\one}{\boldsymbol{1}}

\newcommand{\gen}{\mathsf{gen}}
\newcommand{\supp}{\mathsf{supp}}

\renewcommand{\d}{\,\mathsf{d}}

\newcommand{\intd}{\,\mathsf{d}}
\newcommand{\M}{\mathcal{M}}

\definecolor{ggreen}{RGB}{0,200,0}
\definecolor{rred}{RGB}{150,0,70}
\definecolor{yyellow}{RGB}{250,210,0}

\author{By~~{\scshape Bryna Kra}~~and~~{\scshape Joel~Moreira}~~and~~{\scshape Florian~K.~Richter}\\~~and~~{\scshape Donald Robertson}}
\date{\small \today}
\title{\textbf{A proof of \Erdos{}'s $B+B+t$ conjecture}}

\begin{document}

\maketitle

\begin{abstract}
We show that every set $A$ of natural numbers with positive upper density can be shifted to contain the restricted sumset
$\{b_1 + b_2 : b_1, b_2\in B \text{ and } b_1 \ne b_2 \}$ for some infinite set $B \subset A$.
\end{abstract}

%\tableofcontents

\section{Introduction}
Looking to connect two major achievements of additive combinatorics from the 1970s, Hindman's theorem~\cite{Hindman-1974} and \Szemeredi{}'s theorem~\cite{S},
\Erdos{} formulated the following conjecture on multiple occasions.

\begin{Conjecture}[\Erdos{}~{\cite[Page~305]{erdos-1975},  \cite[Pages 57--58]{erdos-1977}, and~\cite[Page 105]{erdos-1980}}]
\label{conj_main}
For any $A\subset\N$ with positive density there exists an infinite set $B\subset A$ and a number $t\in\N$ such that \[A-t\supset \{b_1+b_2: b_1, b_2\in B \text{ and } b_1\neq b_2\}.\]
\end{Conjecture}

This problem was studied by various authors, including Nathanson~\cite{Nathanson}, Kazhdan (see~\cite{Nathanson, Hindman-1979}), and Hindman~\cite[Section 11]{Hindman-1979}.
The latter provided several equivalent forms, including a natural reformulation using the Stone-\Cech{} compactification of the integers.
A special case of \cref{conj_main}, also conjectured by \Erdos{}, was resolved in \cite{MRR}, asserting that, under the same assumptions, $A$ contains a sumset $B+C=\{b+c:b\in B,c\in C\}$ of two infinite sets $B,C\subset\N$.
Further recent progress in this direction has been made in~\cite{DGJLLM15, Host, KMRR}, and further history on \cref{conj_main} and surrounding problems can be found in~\cite{Hindman-1979,MRR,Nathanson}.

Our main theorem resolves \cref{conj_main}.
To state our result precisely, recall that a \define{\Folner{} sequence} $\Phi$ on $\N$ is any sequence $N \mapsto \Phi_N$ of finite subsets of $\N$ with the property that
\[
\lim_{N \to \infty} \dfrac{|\Phi_N \cap (\Phi_N + t)|}{|\Phi_N|} = 1
\]
for all $t \in \N$. A set $A \subset \N$ has \define{positive upper Banach density} if
\[
\lim_{N \to \infty} \dfrac{|A \cap \Phi_N|}{|\Phi_N|} > 0
\]
for some \Folner{} sequence $\Phi$.

\begin{Theorem}
\label{thm_main}
For any $A\subset\N$ with positive upper Banach density, the following hold:
\begin{enumerate}
[label=(\roman{enumi}),ref=(\roman{enumi}),leftmargin=*]
\item\label{itm_main_thm_i}
There exist an infinite set $B\subset A$ and a shift $t\in\N$ such that
\[
\{b_1+b_2: b_1, b_2\in B \text{ and } b_1\neq b_2\}\subset A-t.
\]
\item\label{itm_main_thm_ii}
There exist an infinite set $B\subset \N$ and a shift $t\in\N$ such that
\[
B\cup
\{b_1+b_2: b_1, b_2\in B \text{ and } b_1\neq b_2\}\subset A-t.
\]
\end{enumerate}
\end{Theorem}

Note that in the formulation of \cref{thm_main} it is not possible to omit the shift by $t$ or remove the condition $b_1 \ne b_2$ in either conclusion (see the discussion in~\cite{MRR} after Question~6.2).
Also, it was observed by Hindman in \cite{Hindman-1979} that writing $t=2r+s$ for $r\in\N$ and $s\in\{0,1\}$ and replacing $B$ by $B-r$, one obtains the following corollary from \cref{thm_main}.

\begin{Corollary}
For any set $A$ of even integers with positive upper Banach density there exists an infinite set $B\subset\N$ such that $A\supset  \{b_1+b_2: b_1, b_2\in B \text{ and } b_1\neq b_2\}.$
\end{Corollary}

Our proof of \cref{thm_main} uses ergodic theory and builds on the new dynamical methods developed in~\cite{KMRR}
to find infinite patterns in sets with positive upper density.
To formulate our main dynamical result we recall some basic terminology.
By a \define{topological system}, we mean a pair $(X,T)$ where $X$ is a compact metric space and $T\colon  X \to X$ is a homeomorphism.
A \define{system} is a triple $(X,\mu,T)$, where $(X,T)$ is a topological system and $\mu$ is a $T$-invariant Borel probability measure on $X$.
The system is \define{ergodic} if any $T$-invariant Borel subset of $X$ has either measure $0$ or measure $1$, and equivalently we say that $\mu$ is ergodic for $T$.
Given a system $(X,\mu,T)$, a point $a \in X$ is \define{generic} for $\mu$ along a \Folner{} sequence $\Phi$, written $a \in \gen(\mu,\Phi)$, if
\[
\mu=\lim_{N\to\infty}\frac1{|\Phi_N|}\sum_{n\in\Phi_N}\delta_{T^na}
\]
where $\delta_x$ is the Dirac measure at $x \in X$ and the limit is in the weak* topology.
This allows us to formulate a dynamical result equivalent to \cref{thm_main}.

\begin{Theorem}
\label{thm_maindynamical}
Let $(X,\mu,T)$ be an ergodic system, let $a\in\gen(\mu,\Phi)$ for some \Folner{} sequence $\Phi$, and let $E\subset X$ be an open set with $\mu(E)>0$.
\begin{enumerate}
[label=(\roman{enumi}),ref=(\roman{enumi}),leftmargin=*]
\item\label{itm_maindynamical_i}
There exist $x_1, x_2\in X$, $t\in\N$, and a strictly increasing sequence $n_1<n_2<\ldots$ of integers such that $x_1\in E$, $T^tx_2\in E$, and $(T\times T)^{n_i}(a, x_1)\to (x_1, x_2)$ as $i\to\infty$.
\item\label{itm_maindynamical_ii}
There exist $x_1, x_2\in X$, $t\in\N$, and a strictly increasing sequence $n_1<n_2<\ldots$ of integers such that $(T\times T)^t (x_1,x_2)\in E\times E$ and $(T\times T)^{n_i}(a, x_1)\to (x_1, x_2)$ as $i\to\infty$.
\end{enumerate}
\end{Theorem}

A proof of the equivalence between
\cref{thm_main} and \cref{thm_maindynamical} is given in \cref{sec_dyn_reduction}, and the proof of \cref{thm_maindynamical} is given in \cref{sec_main_dynamical}.
For a comparison between the techniques in~\cite{KMRR} and this paper, and an outline of how the new difficulties arising are overcome, see \cref{sec:outline-of-proof}.

We conclude the introduction with a natural conjecture on a higher order version of our main theorem.

\begin{Conjecture}
Let $A\subset\N$ have positive upper Banach density and let $k\in\N$.
Then there exist an infinite set $B\subset \N$ and a shift $t\in\N$ such that
\begin{equation}\label{eq_conjecure}
A-t\supset \left\{\sum_{n\in F}n:F\subset B,\ 0<|F|< k\right\}.
\end{equation}
\end{Conjecture}
We remark that an example of Straus answering an earlier question of \Erdos{} (see~\cite[Theorem 2.2]{BBHS} and~\cite[Theorem 11.6]{Hindman-1979}) shows that $k$ can not be replaced by infinity in \eqref{eq_conjecure}.

\paragraph{Acknowledgements:}
BK acknowledges National
Science Foundation grant DMS-2054643, JM and FKR thank the organizers of the conference ``Ultramath2022'' during which part of this project was completed, and DR acknowledges EPSRC grant V050362.
We thank the anonymous referee for helpful suggestions and comments.

\section{Reduction to a dynamical statement}
\label{sec_dyn_reduction}
In this section we show the equivalence between Theorems \ref{thm_main} and \ref{thm_maindynamical}, beginning with the easier implication.
\begin{proof}[Proof that \cref{thm_main} implies \cref{thm_maindynamical}]
We prove that part~\ref{itm_main_thm_i} of \cref{thm_main} implies part~\ref{itm_maindynamical_i} of \cref{thm_maindynamical}; the same proof with obvious modifications shows that part~\ref{itm_main_thm_ii} of \cref{thm_main} implies part~\ref{itm_maindynamical_ii} of \cref{thm_maindynamical}.

    Let $(X,\mu,T)$ be an ergodic system, let $a\in\gen(\mu,\Phi)$ for some \Folner{} sequence $\Phi$, and let $E\subset X$ be an open set with $\mu(E)>0$.
    Since $E$ is open, there exists some point $y\in E$ that lies in the support of $\mu$.
    Let $U$ be an open ball centered at $y$ (and so $\mu(U)>0$) whose closure is contained in $E$.
    Since $a\in\gen(\mu,\Phi)$, the set $A:=\{n\in\N:T^na\in U\}$ has positive upper Banach density.
    Part \ref{itm_main_thm_i} of \cref{thm_main} then implies that $A\supset \{b_1+b_2:b_1,b_2\in B, b_1\neq b_2\}+t$ for some infinite set $B\subset A$ and some $t\in\N$.
Compactness of $X$ yields an increasing sequence $(n_i)_{i\in\N}$, taking values in $B$ for which $x_1:=\lim_{i\to\infty}T^{n_i}a$ exists.
Passing to a subsequence of $(n_i)_{i\in\N}$ if needed, the limit $x_2:=\lim_{i\to\infty}T^{n_i}x_1$ also exists.
Since $B\subset A$, it follows that $x_1\in \overline{U}\subset E$. Since $\{b_1+b_2:b_1,b_2\in B, b_1\neq b_2\}+t\subset A$, we have that $T^tx_2\in E$ as well.
\end{proof}

When the set $A$ in \cref{thm_main} is of the form
\[
A=\{ n \in \N : \|\theta + n \alpha \|_{\R / \Z} < \epsilon \}
\]
for some $\epsilon > 0$ (or, more generally, is a Bohr set),
the existence of a set $B \subset \N$ satisfying the conclusion of \cref{thm_main} is connected to the behavior of $3$-term arithmetic progressions $\theta,\theta+\beta,\theta+2\beta$ in $\R/\Z$ (or, more generally, in the underlying group).
For arbitrary $A \subset \N$ we bridge the gap between the combinatorial statement \cref{thm_main} and the dynamical statement \cref{thm_maindynamical} using a dynamical variant of $3$-term progressions defined as follows.

\begin{Definition}
Given a topological system $(X,T)$, a point $(x_0,x_1,x_2)\in X^3$ is called a \emph{(3-term) \Erdos{} progression} if there exists a strictly increasing sequence $n_1<n_2<\cdots$ of integers such that $(T\times T)^{n_i}(x_0, x_1)\to (x_1, x_2)$ as $i\to\infty$.
\end{Definition}

The role played in this paper by \Erdos{} progressions parallels the role played by \Erdos{} cubes in~\cite{HK-05}.
Various other notions of dynamical progressions, for example those in~\cite{Furstenberg-1977, HK-05, GHSY}, have already been used for related questions, but the one we use does not seem to have been defined previously.
We remark that in group rotations all the notions of dynamical progressions agree with the conventional notion of arithmetic progression.

The next result completes the translation between ergodic theory and combinatorics by connecting \Erdos{} progressions and sumsets.

\begin{Theorem}
\label{thm_erdos-triangle_gives_sumset}
Fix a topological system $(X,T)$ and open sets $U,V \subset X$.
If there exists an \Erdos{} progression $(x_0,x_1,x_2) \in X^3$ with $x_1\in U$ and $x_2 \in V$, then
there exists some infinite set $B \subset \{ n \in \N : T^n x_0 \in U \}$ such that $\{ b_1 + b_2 : b_1,b_2 \in B, b_1 \ne b_2 \}$ is a subset of $\{n \in \N : T^n x_0 \in V \}$.
\end{Theorem}

\begin{proof}
Let $c\colon \N\to\N$ be a strictly increasing sequence such that $(T\times T)^{c(n)}(x_0,x_1)\to(x_1,x_2)$.
Since $U$ is a neighborhood of $x_1$, by refining the sequence $c(n)$ we can assume without loss of generality that $\{c(n): n\in\N\}\subset \{n \in \N : T^n x_0 \in U\}$.

We now construct the set $B\subset \{c(n): n\in\N\}$ inductively.
First choose $b(1)$ in $\{c(i) : i \in \N \}$ with $T^{b(1)} x_1 \in V$.
Note that with this choice of $b(1)$ the set $(T^{-b(1)} V )\times V$ is a neighborhood of $(x_1,x_2)$.
Next, choose $b(2)$ in $\{c(i) : i \in \N \}$ with $b(2) > b(1)$ and
\[
(T \times T)^{b(2)} (x_0,x_1) \in \bigl(T^{-b(1)} V\bigr) \times V.
\]
It follows that $T^{b(1)+b(2)} x_0 \in V$ and $x_1\in T^{-b(2)} V \cap T^{-b(1)} V$.

Supposing that, by induction, we have found $b(1) < \cdots < b(n)\subset \{c(n): n\in\N\}$ with
\[
x_0 \in \bigcap_{1 \le i < j \le n} T^{-b(i) - b(j)} V
\qquad
\text{and}
\qquad
x_1 \in \bigcap_{1 \le i \le n} T^{-b(i)} V,
\]
we choose $b(n+1)\in\{c(i):i\in\N\}$ with $b(n+1) > b(n)$ and
\[
(T \times T)^{b(n+1)} (x_0,x_1) \in \biggl( \bigcap_{1 \le i \le n} T^{-b(i)} V\biggr) \times V.
\]
This is possible because
\[
\bigl( \bigcap_{1 \le i \le n} T^{-b(i)} V\bigr) \times V
\]
is a neighborhood of $(x_1,x_2)$ and $(T\times T)^{c(n)}(x_0,x_1)\to(x_1,x_2)$ as $n\to\infty$.
Together with the inductive hypothesis, this implies
\[
x_0 \in \bigcap_{1 \le i < j \le n+1} T^{- b(i) -b(j)} V
\qquad
\text{and}
\qquad
x_1 \in \bigcap_{1 \le i \le n+1} T^{- b(i)} V
\]
concluding the induction.
Taking $B = \{ b(i) : i \in \N \}$ finishes the proof.
\end{proof}

To deduce \cref{thm_main} from \cref{thm_maindynamical} using \cref{thm_erdos-triangle_gives_sumset}, we use the following version of the Furstenberg correspondence principle.

\begin{proposition}[{\cite[Theorem 2.10]{KMRR}}]\label{lemma_correspondence}
Given a set $A\subset\N$ with positive upper Banach density there exists an ergodic system $(X,\mu,T)$, a F\o lner sequence $\Phi$, a point $a\in\gen(\mu,\Phi)$, and a clopen set $E\subset X$ such that $\mu(E)>0$ and $A=\{n\in\N:T^na\in E\}$.
\end{proposition}

\begin{proof}[Proof that \cref{thm_maindynamical} implies  \cref{thm_main}]

Suppose $A\subset\N$ has positive upper Banach density.
Invoking \cref{lemma_correspondence} we find an ergodic system $(X,\mu,T)$, a point $a\in \gen(\mu, \Phi)$, a \Folner{} sequence $\Phi$, and a clopen set $E\subset X$ such that $\mu(E)>0$ and $A=\{n\in\N:T^na\in E\}$.
Using \cref{thm_maindynamical}, part~\ref{itm_maindynamical_i},  we can find $t\in\N$ and an \Erdos{} progression of the form $(a,x_1,x_2)\in X^3$ such that $x_1\in E$ and $x_2\in T^{-t}E$.
It now follows from \cref{thm_erdos-triangle_gives_sumset}, applied with $U=E$ and $V=T^{-t}E$, that there exists an infinite set $B\subset\{n\in\N:T^na\in E\}=A$ such that
\[
A-t=\{n\in\N: T^n a\in T^{-t}E\}\supset \{b_1+b_2:b_1,b_2\in B,b_1\neq b_2\},
\]
completing the proof of \cref{thm_main}, part~\ref{itm_main_thm_i}.
If we invoke part~\ref{itm_maindynamical_ii} of \cref{thm_maindynamical} instead, then the same argument, but using  \cref{thm_erdos-triangle_gives_sumset} applied with $U=V=T^{-t}E$, yields \cref{thm_main}, part~\ref{itm_main_thm_ii}.

\end{proof}

\section{Proof of the dynamical statement ({\cref{thm_maindynamical}})}
\label{sec_main_dynamical}

\subsection{Outline of the proof}
\label{sec:outline-of-proof}
Our first observation is that if $(x_0,x_1)$ is generic for a $(T \times T)$-invariant measure on $X\times X$ and $(x_1,x_2)$ is in the support of that measure, then $(x_0,x_1,x_2)$ forms an \Erdos{} progression.
One may be tempted, then, to find pairs having these properties with respect to the product measure $\mu\times \mu$ on $X\times X$.
However, it may be the case that the product measure is not $T \times T$-ergodic, and so typical pairs may only be generic for one of the ergodic components of $\mu \times \mu$.
This possibility leads us to consider the ergodic decomposition of $\mu \times \mu$.

We use the notation $(x_1,x_2)\mapsto\lambda_{(x_1,x_2)}$ to denote an ergodic decomposition of $(X\times X,\mu\times\mu,T\times T)$.
Since we aim to produce an \Erdos{} progression with prescribed first coordinate $a$, as in~\cite{KMRR} we make use of a decomposition $\lambda_{(x_1,x_2)}$ that is defined for \textit{every} pair $(x_1,x_2)$ and continuous as a function of $(x_1,x_2)$.
We show in the next section that, without loss of generality, we may assume that our system admits such a decomposition, referred to as a \textit{continuous ergodic decomposition}.

To find an \Erdos{} progression $(a,x_1,x_2)\in X\times X\times X$,  it then suffices to find a pair $(x_1,x_2)$ with all of the following properties.
\begin{enumerate}
\item
The pair $(a,x_1)\in X\times X$ is generic for the measure $\lambda_{(a,x_1)}$.
\item
The pair $(x_1,x_2)\in X\times X$ lies in the support of $\lambda_{(x_1,x_2)}$.
\item
The measures $\lambda_{(a,x_1)}$ and $\lambda_{(x_1,x_2)}$ are equal.
\end{enumerate}
While Properties~1 and~2 hold for $(\mu \times \mu)$-almost every point $(x_1,x_2)\in X\times X$, the explicit construction of a continuous ergodic decomposition in~\cite{KMRR} tells us that Property~3 holds if and only if the triple $(a,x_1,x_2)$ sits above a three-term progression in the \textit{Kronecker factor} of $(X,\mu,T)$. Since $a$ is fixed, this constitutes a set of zero measure with respect to $\mu \times \mu$ whenever the Kronecker factor is not finite.

We must therefore show that Properties~1 and~2 hold within the set of points $(x_1,x_2)$ for which $(a,x_1,x_2)$ projects onto a three-term progression in the Kronecker factor.
To that end, we introduce a natural measure $\sigma$ on $X\times X$ giving full measure to the set of points $(x_1,x_2)$ such that $(a,x_1,x_2)$ sits above a three-term progression. Thus $\sigma$ is a measure with the property that almost every pair $(x_1,x_2)\in X\times X$ satisfies Property~3.
Most of the work then goes into showing that the first two properties hold for $\sigma$-almost every pair $(x_1,x_2)$.  This is where the present work diverges from~\cite{KMRR}. To establish the properties we want, it is necessary to understand in greater detail than~\cite{KMRR} the disintegration of $\mu$ over the Kronecker factor.

\subsection{Using continuous factor maps}
\label{sec:factors}

Throughout this section, we make use of two types of factor maps from a system $(X,\mu,T)$ to another system $(Y,\nu,S)$.
\begin{itemize}
\item
\emph{Measurable factor maps}: a measurable function $\pi\colon X\to Y$ such that $\pi(\mu)=\nu$ and $\pi\circ T=S\circ\pi$ $\mu$-almost everywhere.
\item
\emph{Continuous factor maps}: a continuous surjection $\pi\colon X\to Y$ that such that $\pi(\mu) = \nu$ and $\pi \circ T = S \circ \pi$ everywhere.
\end{itemize}
If there exists a measurable factor map $\pi\colon X\to Y$, then $(Y,\nu,S)$ is called a \define{factor} of $(X,\mu,T)$.

In his proof of \Szemeredi{}'s theorem, Furstenberg~\cite{Furstenberg-1977} shows that in order to understand the behavior of $3$-term dynamical progressions, it suffices to consider their projections onto the maximal group rotation factor.  We use an analogous method to study  $3$-term \Erdos{} progressions.

A \emph{group rotation} is a system of the form $(Z,m,R)$, where $Z$ is a compact abelian group, $m$ is the Haar measure on $Z$, and $R\colon Z\to Z$ is a rotation of the form $R(z) = z+\alpha$ for a fixed element $\alpha\in Z$.
Whenever $(Z,m,R)$ is a group rotation, we assume that the metric on $Z$ is chosen such that $z \mapsto z + w$ is an isometry for all $w \in Z$.

Every ergodic system $(X,\mu,T)$ possesses a maximal
group rotation factor called its \define{Kronecker factor} (see {\cite[Section~3]{Furstenberg-book}}).
In general, the factor map from an ergodic system $(X,\mu,T)$ onto its Kronecker factor $(Z,m,R)$ is only a measurable factor map.
The next lemma, however, shows that in many situations one can assume without loss of generality that the factor map onto the Kronecker factor is continuous, and this is key in our proof of \cref{thm_maindynamical}.

\begin{Proposition}[{\cite[Proposition 3.20]{KMRR}}]
\label{lemma_continuouskroneckerenabler}
Let $(X,\mu,T)$ be an ergodic system and let $a\in\gen(\mu,\Phi)$ for some \Folner{} sequence $\Phi$.
Then there exists an ergodic system $(\tilde{X},\tilde{\mu},\tilde{T})$, a \Folner{} sequence $\Psi$, a point $\tilde{a} \in\tilde{X}$ and a continuous factor map $\tilde\pi\colon \tilde X\to X$ such that $\tilde\pi(\tilde a)=a$ and $\tilde{a} \in \gen(\tilde{\mu},\Psi)$ and $(\tilde X,\tilde\mu,\tilde T)$ has a continuous factor map to its Kronecker factor.
\end{Proposition}

With the help of \cref{lemma_continuouskroneckerenabler} we can reduce the proof of \cref{thm_maindynamical} to the following special case.

\begin{Theorem}
\label{thm_maindynamicalcontinuouskronecker}
Let $(X,\mu,T)$ be an ergodic system and assume there is a continuous factor map $\pi$ to its Kronecker.
Let $a\in\gen(\mu,\Phi)$ for some \Folner{} sequence $\Phi$, and let $E\subset X$ be a Borel set with
$\mu(E)>0$.
\begin{enumerate}
[label=(\roman{enumi}),ref=(\roman{enumi}),leftmargin=*]
\item\label{itm_maindynamicalcontinuouskronecker_i}
There exist $t\in\N$ and an \Erdos{} progression of the form $(a,x_1,x_2)\in X^3$ such that $x_1\in E$ and $T^tx_2\in E$.
\item\label{itm_maindynamicalcontinuouskronecker_ii}
There exist $t\in\N$ and an \Erdos{} progression of the form $(a,x_1,x_2)\in X^3$ such that $T^tx_1\in E$ and $T^tx_2\in E$.
\end{enumerate}
\end{Theorem}

%\new{\begin{Remark}
   We remark that unlike in \cref{thm_maindynamical}, in the formulation of \cref{thm_maindynamicalcontinuouskronecker} we do not require that $E$ is an open set.
    In fact, this hypothesis is not needed in \cref{thm_maindynamical} either, but without assuming openness of $E$, \cref{thm_maindynamical} is no longer equivalent to \cref{thm_main}.
%\end{Remark}}

\begin{proof}[Proof that \cref{thm_maindynamicalcontinuouskronecker} implies \cref{thm_maindynamical}]

We only prove that part~\ref{itm_maindynamicalcontinuouskronecker_i} of \cref{thm_maindynamicalcontinuouskronecker} implies part~\ref{itm_maindynamical_i} of \cref{thm_maindynamical}.
Similar arguments show the implication between part~\ref{itm_maindynamicalcontinuouskronecker_ii} of \cref{thm_maindynamicalcontinuouskronecker} and part~\ref{itm_maindynamical_ii} of \cref{thm_maindynamical}.

Let $(X,\mu,T)$ be an ergodic system, let $a\in\gen(\mu,\Phi)$ for some F\o lner sequence $\Phi$ and let $E\subset X$ be open and have positive measure.
Let $(\tilde{X},\tilde{\mu},\tilde{T})$,  $\tilde{a}$ and $\tilde{\pi}$ result from an application of \cref{lemma_continuouskroneckerenabler} and let $\tilde E:=\tilde\pi^{-1}(E)\subset \tilde X$.
As $(\tilde{X},\tilde{\mu},\tilde{T})$ has a continuous factor map to its Kronecker factor, we can apply \cref{thm_maindynamicalcontinuouskronecker} to find $t\in\N$ and an \Erdos{} progression $(\tilde{a},\tilde{x}_1,\tilde{x}_2)\in\tilde X^3$ with $\tilde x_1\in\tilde E$ and $\tilde{x}_2\in T^{-t}\tilde{E}$.
It is then immediate that $(\tilde{\pi}(\tilde{a}),\tilde{\pi}(\tilde{x}_1),\tilde{\pi}(\tilde{x}_2))$ is an \Erdos{} progression in $X^3$ with $\tilde{\pi}(\tilde{x}_1)\in E$ and $\tilde{\pi}(\tilde{x}_2)\in T^{-t}(E)$.
\end{proof}

The proof of \cref{thm_maindynamicalcontinuouskronecker} is deferred to \cref{subsec_maindynamical} until after we have developed the necessary tools.
We conclude this section by recalling the continuous ergodic decomposition of the product measure $\mu \times \mu$ from~\cite{KMRR}. To do so we make use of the following standard disintegration result.

\begin{Theorem}[See {\cite[Theorem~5.14]{EW11}}]
\label{thm_cmd}
Given a measurable factor map $\pi\colon  X\to Y$ between systems $(X,\mu,T)$ and $(Y,\nu,S)$, there exists a measurable map $y \mapsto \mu_y$ defined on a full measure subset of $Y$ and taking values in the space $\M(X)$ of Borel probability measures on $X$ with the following properties.
\begin{enumerate}[label=(\roman{enumi}),ref=(\roman{enumi}),leftmargin=*]
\item\label{itm_cm_i}
For every bounded, measurable function $f\colon X\to \R$, the function
\[
y\mapsto \int_X f\d\mu_y
\]
is an almost everywhere defined and Borel measurable function on $Y$ satisfying
\[
\int_{D}\biggl(\int_X f\d\mu_y\biggr) \d\nu(y)=\int_{\pi^{-1}(D)} f\d\mu
\]
for all Borel sets $D\subseteq Y$.
\item\label{itm_cm_ii}
For $\nu$-almost every $y\in Y$, we have $\mu_y(\pi^{-1}(\{y\}))=1$.
\item\label{itm_cm_iii}
Properties \ref{itm_cm_i} and \ref{itm_cm_ii} uniquely determine the map $y\mapsto \mu_y$ in the sense that if $y\mapsto \mu_y'$ is another measurable map from $Y$ to $\mathcal{M}(X)$ with these properties, then $\mu_y=\mu_y'$ for $\nu$-almost every $y\in Y$.
\item\label{itm_cm_iv}
For almost every $y\in Y$, we have $T \mu_y = \mu_{S y}$.
\end{enumerate}
\end{Theorem}

Fix an ergodic system $(X,\mu,T)$.
Let $(Z,m,R)$ be its Kronecker factor, and assume that $\pi$ is a continuous factor map from $(X,\mu,T)$ to $(Z,m,R)$.
Also fix a disintegration $z\mapsto\eta_z$ of $\mu$ with respect to $\pi$.
As in~\cite[Equation (3.10)]{KMRR}, for every $(x_1,x_2) \in X\times X$ we define
\begin{equation}
\label{eq:lambda-defined}
\lambda_{(x_1,x_2)} = \int_Z \eta_{z + \pi(x_1)} \times \eta_{z + \pi(x_2)} \intd m(z)
\end{equation}
on $X\times X$.
Note that $\lambda_{(x_1,x_2)}$ does not depend on the choice of disintegration $z\mapsto\eta_z$.
The following properties are proved in \cite[Proposition 3.11]{KMRR}.
\begin{enumerate}
\item
The map $(x_1,x_2) \mapsto \lambda_{(x_1,x_2)}$ is continuous.
\item
The map $(x_1,x_2) \mapsto \lambda_{(x_1,x_2)}$ is a disintegration of $\mu \times \mu$ in the sense that
\[\int_{X\times X}\lambda_{(x_1,x_2)} \d(\mu\times\mu)(x_1,x_2)=\mu\times\mu
\]
holds.
\item
\label{item:generic-lambda}
For $(\mu \times \mu)$-almost every $(x_1,x_2)$, the point $(x_1,x_2)$ is generic for $\lambda_{(x_1,x_2)}$ and $\lambda_{(x_1,x_2)}$ is $T \times T$ ergodic.
\item
For every $(x_1,x_2) \in X\times X$, we have that
$\lambda_{(x_1,x_2)} = \lambda_{(Tx_1, Tx_2)}$.
\end{enumerate}

\subsection{The measure on \Erdos{} progressions}
\label{sec_measure_on_progressions}

In this section, we introduce a measure $\sigma$ on $X \times X$ that helps us study  \Erdos{} progressions beginning at a fixed point $a \in X$.

Fix an ergodic system $(X,\mu,T)$ and a point $a\in X$, let $(Z,m,R)$ denote its Kronecker factor, and further assume that there is a continuous factor map $\pi\colon X\to Z$.
Moreover, we fix a disintegration $z\mapsto\eta_z$ of $\mu$ over $\pi$ as guaranteed by \cref{thm_cmd}.

\begin{Definition}
\label{def_sigma}
We define the measure
\begin{equation}
\label{eq_triangularsigmas}
\sigma=\int_{Z}\eta_{z}\times\eta_{2z-\pi(a)}\intd m(z)=\int_{Z}\eta_{\pi(a)+z}\times\eta_{\pi(a)+2z}\d m(z)
\end{equation}
on $X \times X$.
\end{Definition}

For the remainder of this paper, we  use $\sigma$ to denote the measure defined by \eqref{eq_triangularsigmas}.
Note that the second equality in~\eqref{eq_triangularsigmas} follows from translation invariance of $m$.
We stress that $\sigma$ does not depend on the exact choice of disintegration $z\mapsto \eta_z$ since any two choices agree $m$-almost everywhere.

The motivation for this definition is that $\sigma$ is a relatively independent joining, putting as unbiased as possible a measure on the set of pairs $(x_1,x_2)\in X\times X$ such that
\[
\bigl((\pi(a),\pi(x_1),\pi(x_2)\bigr)
\]
forms a $3$-term arithmetic progression in $Z$.
The connection to three-term progressions is made apparent by the equality
\[
\pi(x_2) - \pi(x_1) = \pi(x_1) - \pi(a),
\]
which holds for $\sigma$-almost every $(x_1,x_2)$ and guarantees via~\eqref{eq:lambda-defined} that $\lambda_{(a,x_1)} = \lambda_{(x_1,x_2)}$.

 We conclude this section with some lemmas that are  of use in the next sections.

\begin{lemma}\label{lemma_projectionsofsigma}
Let $\pi_1\colon X\times X\to X$ denote the projection $(x_1,x_2)\mapsto x_1$ onto the first coordinate.
Then $\pi_1\sigma=\mu$.
\end{lemma}
\begin{proof}
For any $f\in C(X)$,  we have
\begin{align*}
\int_{X} f \intd(\pi_1\sigma)
&=
\int_{X\times X} (f\otimes 1) \intd\sigma
\\
&=
\int_{Z} \biggl(\int_{X\times X} (f\otimes 1)
\d(\eta_{z}\times\eta_{2z-\pi(t)})\biggr) \intd m(z)
\\
&=
\int_{Z} \biggl(\int_{X} f
\d\eta_{z}\biggr) \intd m(z)
=
\int_{X} f
\d\mu,
\end{align*}
as desired.
\end{proof}

\begin{lemma}
\label{sigma_pi2_covers_mu}
Let $\pi_2\colon X\times X\to X$ denote the projection $(x_1,x_2)\mapsto x_2$ onto the second coordinate.
Then $\frac{1}{2}( \pi_2 \sigma + T \pi_2 \sigma) = \mu$.
\end{lemma}

\begin{proof}
Denote by $2Z$ the subgroup $\{z+z : z \in Z \}$ and let $\xi$ denote its Haar measure.
Ergodicity of $R$ ensures that $Z=(2Z)\cup R(2Z)$ and that $m=\tfrac12(\xi+R\xi)$.
In particular, for each $s\in X$ there exists $w\in Z$ such that either $\pi(s)=2w$ or $\pi(s)=R(2w)$.
In the first case
\[
\pi_2 \sigma = \int_Z \eta_{2(w+z)} \intd m(z) = \int_Z \eta_{2z} \intd m(z) = \int_{2Z} \eta_u \intd \xi(u),
\]
and in the second
\[
\pi_2 \sigma = \int_Z \eta_{2(w+z) + \alpha} \intd m(z) = \int_Z \eta_{2z + \alpha} \intd m(z) = \int_{2Z + \alpha} \eta_u \intd (R\xi)(u).
\]
Since $T\eta_u=\eta_{Ru}$ and $R^2\xi=\xi$, it follows that in either case
\[
\frac{1}{2}( \pi_2 \sigma + T \pi_2 \sigma) =\int_Z\eta_z\d\tfrac12(\xi+R\xi)(z)= \mu. \quad \qedhere
\]
\end{proof}

\subsection{The support of the measure}
\label{sec_support}

We maintain the notation of \cref{sec_measure_on_progressions} and assume that $(X, \mu, T)$ is an ergodic system with Kronecker factor $(Z, m, R)$, continuous factor map $\pi\colon (X, \mu, T)\to (Z, m, R)$, the measures $\lambda_{x_1, x_2}$ are those defined in~\eqref{eq:lambda-defined}, and $\sigma$ denotes the measure defined in~\eqref{def_sigma}.  We continue to use the fixed disintegration $z\mapsto \eta_z$  of $\mu$ with respect to $\pi$.

\begin{lemma}
\label{lemma_samelambda}
For $\sigma$-almost every $(x_1,x_2)\in X\times X$, the measures $\lambda_{(a,x_1)}$ and $\lambda_{(x_1,x_2)}$ are equal.
\end{lemma}
\begin{proof}
Consider the set
\[
P:=\{(x_1,x_2)\in X\times X
:
\pi(x_1)=\pi(a)+z,~\pi(x_2)=\pi(a)+2z \textup{ for some } z \in Z \}.
\]
Combining~\eqref{eq_triangularsigmas} and property \ref{itm_cm_ii} of \cref{thm_cmd} for the disintegration $z\mapsto\eta_z$, it follows that $\sigma(P)=1$ and each $(x_1,x_2)\in P$ satisfies
\[
\pi(x_2) -  \pi(x_1) =  \pi(x_1)- \pi(a).
\]
Thus we have $\lambda_{(x_1,x_2)} = \lambda_{(a,x_1)}$ by the defining formula~\eqref{eq:lambda-defined} and translation invariance of $m$.
\end{proof}

Let $\supp(\nu)$ denote the support of a Borel measure $\nu$ and let $\F(X)$ denote the family of closed, nonempty subsets of a given compact metric space $(X,d)$.
We endow $\F(X)$ with the Haudsorff metric $\Hau$, defined by
\[
\Hau(F,G) = \max \left\{ \sup_{x \in F} d(x,G),~ \sup_{y \in G} d(y,F) \right\}
\]
whenever $F,G\in\F(X)$.

\begin{lemma}
\label{lem_measurable}
Let $W$ be a compact metric space, $\M(W)$ the space of Borel probability measures on $W$ endowed with the weak* topology, and $\F(W)$ the space of closed, non-empty subsets of $W$ endowed with the Hausdorff metric.
\leavevmode
\begin{enumerate}
\item
\label{item:one-measurable}
The map $\nu \mapsto \supp(\nu)$ from $\M(W)$ to $\F(W)$ is Borel measurable.
\item
\label{item:two-measurable}
If $x\mapsto \rho_x$ is a measurable map from $W$ to $\M(W)$, then
$\{ x \in W : x \in \supp(\rho_{x})\}$ is a Borel set.
\end{enumerate}
\end{lemma}

\begin{proof}
\leavevmode
\begin{enumerate}
\item
Combining Theorem~17.14, Lemma~17.5, and Theorem~18.9 in~\cite{AB-book}, the result follows.
\item
The map $\psi_1(x) = \{x\}$ from $W$ to $\F(W)$ is continuous and hence measurable.
By part~\ref{item:one-measurable}, the map $\psi_2(x) = \supp(\rho_{x})$ from $W$ to $\F(W)$ is also measurable.
Thus $\psi(x) = (\psi_1(x),\psi_2(x))$ from $W$ to $\F(W) \times \F(W)$ is measurable.
The set
$
\Omega = \{ (F_1,F_2) \in \F(W) \times \F(W) : F_1 \cap F_2 \ne \emptyset \}
$
is closed, and therefore $\{ x \in W : x \in \supp(\rho_x) \} = \psi^{-1}(\Omega)$ is Borel.
\qedhere
\end{enumerate}
\end{proof}

\begin{lemma}\label{lemma_disintegrationsupport}
The disintegration $z \mapsto \eta_z$ satisfies $\mu(\{ x \in X : x \in \supp(\eta_{\pi(x)}) \}) = 1$.
\end{lemma}

\begin{proof}
Write $G = \{ x \in X : x \in \supp(\eta_{\pi(x)}) \}$, which is Borel measurable by \cref{lem_measurable},  part~\ref{item:two-measurable}. Since
\[
\mu(G) = \int_Z  \eta_z(G) \intd m(z),
\]
it suffices to show $\eta_z(G)=1$ for $m$-almost every $z\in Z$.
By \cref{thm_cmd}, part~\ref{itm_cm_ii}, for almost every $z\in Z$ we have $\eta_z(\pi^{-1}(z))=1$. If $\eta_z(\pi^{-1}(z))=1$, then $\supp(\eta_z) \subset \pi^{-1}(z)$ because continuity of $\pi$ gives that $\pi^{-1}(z)$ is a closed set, and therefore it is a closed set of full measure.
Thus, for $m$-almost every $z\in Z$, we have $\supp(\eta_z) \subset \pi^{-1}(z)$ and hence $\supp(\eta_z)\subset G$. Since $\supp(\eta_z)\subset G$, we have $\eta_z(G) \ge \eta_z(\supp(\eta_z)) =1$ for $m$-almost every $z\in Z$.
\end{proof}

Write
\begin{equation}
\label{eqn_goodsets}
S=\big\{(x_1,x_2)\in X\times X: (x_1,x_2)\in\supp(\lambda_{(x_1,x_2)})\big\}
\end{equation}
and note that part~\ref{item:two-measurable} in \cref{lem_measurable}, together with continuity of $(x_1,x_2) \mapsto \lambda_{(x_1,x_2)}$, implies that $S$ is a  Borel subset of $X \times X$.
Our goal for the remainder of this section is to show that $\sigma(S)=1$ for every $s\in X$ (see \cref{prop:sigma-s-full}).

\begin{proposition}
\label{prop_almost_every_point_is_density_point}
Fix a system $(X,\mu,T)$ and a continuous factor map $\pi$ to its Kronecker factor $(Z,m,T)$.
Also fix a disintegration $z \mapsto \eta_z$ over its Kronecker factor $(Z,m,R)$.
There is a sequence $\delta(j) \to 0$ such that for almost every $x \in X$ the following holds: for every neighbourhood $U$ of $x$ we have
\begin{equation}
\label{eqn:density_point_sequence}
\lim_{j \to \infty} \frac{m\Big(\big\{z\in Z :  \eta_z(U) > 0\big\} \cap B\big(\pi(x),\delta(j)\big)\Big)}{m\Big(B\big(\pi(x),\delta(j)\big)\Big)}=1.
\end{equation}
\end{proposition}

\begin{proof}
Consider the map $\Phi\colon Z\to \F(X)$ given by $\Phi(z)=\supp(\eta_{z})$.
This map is Borel measurable by \cref{lem_measurable} as it is the composition of two Borel measurable functions $z\mapsto \eta_z$ and $\nu\mapsto \supp(\nu)$.
Applying Lusin's theorem~\cite[Theorem~12.8]{AB-book} for every $j\in\N$, there is a closed set $Z_j\subset Z$ with $m(Z_j)>1-2^{-j}$ such that $\Phi|_{Z_j}$ is continuous.
By uniform continuity of $\Phi|_{Z_j}$, there exists a positive number $\delta(j)$ such that for all $z_1,z_2\in Z_j$,
\[
d(z_1,z_2)\leq\delta(j)~\implies~ \Hau\big(\Phi(z_1),\Phi(z_2)\big)<\frac{1}{j}.
\]

Consider the set
\[
K_{j}=\left\{z\in Z_j: m\Big(B\big(z,\delta(j)\big)\cap Z_j\Big)> \left( 1 - \frac{1}{j} \right) m\Big(B\big(z,\delta(j)\big)\Big)\right\}
\]
for each $j\in \N$.
Define
\[
\chi_j(z)=
\frac{1}{m(B(0,\delta(j)))}\int_{Z_j} \one_{B(0,\delta(j))}(w-z)\d{}m(w)
\]
and note that $\chi_j(z)\leq 1$ for all $z\in Z$.
Since translations on $Z$ are isometries, we have
\begin{equation}
\label{eqn_k_ell_intersec}
K_j= Z_j\cap \left\{z\in Z: \chi_j(z)> \left( 1 - \frac{1}{j} \right)\right\}.
\end{equation}
Using Fubini's theorem, we deduce that
\begin{equation}\label{eq_weirdset1}
 \int_{Z} \chi_j(z)\d{}m(z)
=
m(Z_j)
>
1-\frac1{2^j},
\end{equation}
which combined with $\chi_j(z)\leq 1$ implies that
\begin{equation}\label{eq_weirdset2}
  m\left(\left\{z\in Z: \chi_j(z)> \left( 1 - \frac{1}{j} \right)\right\}\right)\geq1-\frac j{2^j}.
\end{equation}
Combining \eqref{eqn_k_ell_intersec} with \eqref{eq_weirdset1} and \eqref{eq_weirdset2},  it follows that $\sum_{j\in\N}m(Z\setminus K_{j})<\infty$.

Let
\[
K
=
\bigcup_{M\geq 1}~ \bigcap_{j\geq M} K_j.
\]
Observe that, by the Borel-Cantelli lemma, $m(K)=1$.
In view of \cref{lemma_disintegrationsupport}, this implies that the set $L := \{ x \in X : x \in \supp(\eta_{\pi(x)}) \} \cap \pi^{-1}(K)$ has $\mu(L)=1$.
To finish the proof it thus suffices to show that any $x\in L$ satisfies \eqref{eqn:density_point_sequence}.

Fix a point $x\in L$ and let $U$ be a neighborhood of $x$.
Let $z=\pi(x)$.
Since $z\in K$ and $U$ is open, there exists $j_0\in\N$ such that for all $j\geq j_0$ we have $z\in K_j$ and $B(x,1/j)\subset U$.
We claim that for all $j\geq j_0$, we have
\begin{equation}
\label{eqn_cl_010}
B\big(z,\delta(j)\big)\cap Z_j\subset H:=\big\{z\in Z: \eta_z(U)>0)\big\}.
\end{equation}
To verify this claim, let $z'\in B(z,\delta(j))\cap Z_j$ be arbitrary.
Since $\Hau(\Phi(z),\Phi(z'))<1/j$ and $x\in \Phi(z)$, there exists $x'\in \Phi(z')$ with $d(x,x')<1/j$.
From $d(x,x')<1/j$ it follows that $x'\in U$ and using $x'\in \Phi(z')$ we conclude $U\cap \Phi(z')\neq\emptyset$.
Since $\Phi(z')=\supp(\eta_{z'})$, it follows that $\eta_{z'}(U)>0$ and hence that $z'\in H$, proving that \eqref{eqn_cl_010} holds, as claimed.

Since $z\in K_j$, it follows from \eqref{eqn_cl_010} and the construction of $K_j$ that
\[
\frac{m\big(H\cap B(z,\delta(j))\big)}{m\big(B(z,\delta(j))\big)}
\geq
1 - \frac{1}{j}
\]
for all $j\geq j_0$.
We conclude that
\[
\lim_{j \to \infty} \frac{m\big(H\cap B(z,\delta(j))\big)}{m\big(B(z,\delta(j))\big)}=1
\]
and the proof is complete.
\end{proof}

\begin{proposition}
\label{prop:sigma-s-full}
The set $S$ defined in~\eqref{eqn_goodsets} satisfies $\sigma(S)=1$.
\end{proposition}

\begin{proof}
Apply \cref{prop_almost_every_point_is_density_point} to get a sequence $\delta(j) \to 0$ with the properties therein.
Let $L$ denote the set of points satisfying \eqref{eqn:density_point_sequence} which has full $\mu$-measure.
 We conclude from \cref{lemma_projectionsofsigma} that $\sigma(L\times X) = 1$ and conclude from \cref{sigma_pi2_covers_mu} that
\[
1 = \mu(L) = \dfrac{\sigma(X\times L) + \sigma(X\times T^{-1} L)}{2},
\]
whence $\sigma(X\times L) = 1$. Thus
\[
\sigma(L \times L) = \sigma( (X \times L) \cap (L \times X) ) = 1.
\]
To prove $\sigma(S) = 1$, it therefore suffices to show $L \times L \subset S$.

Let $(x_1,x_2)\in L\times L$. Let $U_1$ be a neighborhood of $x_1$ and let $U_2$ be a neighborhood of $x_2$. To show $(x_1,x_2)\in S$, we have to verify $\lambda_{(x_1,x_2)}(U_1\times U_2)>0$.
For convenience, write $\beta=\pi(x_2)-\pi(x_1)$. By definition,
\[
\lambda_{(x_1,x_2)}=\int_Z \eta_{z}\times\eta_{z+\beta}\d m (z).
\]
Since $x_1,x_2\in L$, there exists some $\delta>0$ such that
\begin{equation}
    \label{eq:first-coordinate}
\frac{m(\{z\in Z: \eta_z(U_1)>0)\}\cap B(\pi(x_1),\delta))}{m(B(\pi(x_1),\delta))}\geq \frac{3}{4}
\end{equation}
as well as
\begin{equation}
    \label{eq:second-coordinate}
\frac{m(\{z\in Z: \eta_z(U_2)>0)\}\cap B(\pi(x_2),\delta))}{m(B(\pi(x_2),\delta))}\geq \frac{3}{4}.
\end{equation}
Observe that $\{z\in Z: \eta_z(U_2)>0)\}-\beta=\{z\in Z: \eta_{z+\beta}(U_2)>0)\}$, and so \eqref{eq:second-coordinate} implies
\begin{equation}
    \label{eq:second-coordinate_shifted}
\frac{m(\{z\in Z: \eta_{z+\beta}(U_2)>0)\}\cap B(\pi(x_1),\delta))}{m(B(\pi(x_1),\delta))}\geq \frac{3}{4}.
\end{equation}

Define $W=\{z\in Z: \eta_z(U_1)>0~\text{and}~\eta_{z+\beta}(U_2)>0)\}$.
By~\eqref{eq:first-coordinate} and~\eqref{eq:second-coordinate_shifted} it follows that $W$ contains at least one-quarter of the ball $B(\pi(x_1),\delta)$, which implies $m(W)>0$.
Since for all $z\in W$ one has
\[
(\eta_z\times\eta_{z+\beta})(U_1\times U_2)>0.
\]
and $m(W)>0$, it follows that
$\lambda_{(x_1,x_2)}(U_1\times U_2)>0$ as desired.
\end{proof}

\subsection{Proof of \texorpdfstring{\cref{thm_maindynamicalcontinuouskronecker}}{main dynamical theorem}}
\label{subsec_maindynamical}

To prove \cref{thm_maindynamicalcontinuouskronecker} we need one further lemma.

\begin{lemma}[{\cite[Lemma 3.18]{KMRR}}]\label{lemma_optimistic}
Let $(X,\mu,T)$ be an ergodic system, let $a\in\gen(\mu,\Phi)$ for some \Folner{} sequence $\Phi$. Then there exists a \Folner{} sequence $\Psi$ such that for $\mu$-almost every $x_1\in X$ the point $(a,x_1)$ belongs to $\gen(\lambda_{(a,x_1)},\Psi)$.
\end{lemma}
\begin{proof}
From property~\eqref{item:generic-lambda} after the definition of $\lambda_{(x,y)}$ in~\eqref{eq:lambda-defined} and Fubini's theorem, there exists (a full measure set of) $b\in\supp(\mu)$ such that for $\mu$-almost every $x\in X$, the point $(b,x)$ is generic for $\lambda_{(b,x)}$ with respect to the F\o lner sequence $(\{1,\dots,N\})_{N\in\N}$.
    Let $(G_j)_{j=1}^\infty$ enumerate a countable dense subset of $C(X\times X)$ and, for each $j\in\N$, let $\tilde G_j(x,y)=\int_{X^2}G_j\d\lambda_{(x,y)}$.
    Since the map $(x_1,x_2)\mapsto\lambda_{(x_1,x_2)}$ is continuous and $(T\times T)$-invariant, each of the functions $\tilde G_j$ is also continuous and $(T\times T)$-invariant.

    Since $a\in\gen(\mu,\Phi)$ and $b\in\supp(\mu)$, for every $m\in\N$, there exists $s(m)\in\N$ such that
    $\|G_j(b,\cdot)-G_j(T^{s(m)}a,\cdot)\|_\infty<2^{-m}$ for every $j\leq m$ and $\|\tilde G_j(b,\cdot)-\tilde G_j(T^{s(m)}a,\cdot)\|_\infty<{2^{-m}}$ for every $j\leq m$.
    Since $(b,x)$ is generic for $\lambda_{(b,x)}$ for $\mu$-almost every $x\in X$, for each $m\in\N$, there exists some $N(m)\in\N$ for which the continuous function
    $$F_{m}(x):=\max_{1\leq j\leq m}\left|\frac1{N(m)}\sum_{n=1}^{N(m)}G_j(T^nb,T^nx)-\tilde G_j(b,x)\right|$$
    satisfies
    $\|F_m\|_{L^1(\mu)}<1/2^m$.
    The choice of $s(m)$ implies that
    $$\tilde F_{m}(x):=\max_{1\leq j\leq m}\left|\frac1{N(m)}\sum_{n=1}^{N(m)}G_j(T^{n+s(m)}a,T^nx)-\tilde G_j(T^{s(m)}a,x)\right|$$
    satisfies
    $\|\tilde F_m\|_{L^1(\mu)}<3/2^m$.
    Letting $\Psi_m=\{s(m)+1,\dots,s(m)+N(m)\}$ and using $T\times T$ invariance of $\tilde G_j$ we deduce that
    $$F_m'(x):=\tilde F_m(T^{s(m)}x)=\max_{1\leq j\leq m}\left|\frac1{|\Psi_m|}\sum_{n\in\Psi_m}G_j(T^{n}a,T^nx)-\tilde G_j(a,x)\right|.$$
    Since $\mu$ is $T$-invariant, $\|F'_m\|_{L^1(\mu)}=\|\tilde F_m\|_{L^1(\mu)}<3/2^m$ for every $m\in\N$, hence the function $F(x):=\sum_{m\in\N}F_m'(x)$ has norm $\|F\|_{L^1(\mu)}=\sum\|F'_m\|_{L^1(\mu)}<\infty$ and is therefore finite almost everywhere.
    For every point $x_1\in X$ for which $F(x_1)<\infty$, we have that $F_m'(x_1)\to0$ as $m\to\infty$ and hence $(a,x_1)\in\gen(\lambda_{(a,x_1)},\Psi)$.
\end{proof}

\begin{proof}[Proof of \cref{thm_maindynamicalcontinuouskronecker}]
Fix a system $(X,\mu,T)$, $a \in \gen(\mu,\Phi)$ for some \Folner{} sequence $\Phi$, and $E \subset X$ open with $\mu(E) > 0$.
Assume $(X,\mu,T)$ has a continuous factor map $\pi$ to its Kronecker factor $(Z,m,R)$.
Applying \cref{lemma_optimistic} it follows that for $\mu$-almost every $x_1\in X$, we have $(a,x_1)\in\gen(\lambda_{(a,x_1)},\Psi)$ for some \Folner{} sequence $\Psi$.
Since, in view of \cref{lemma_projectionsofsigma}, the projection of $\sigma$ onto the first coordinate equals $\mu$, it follows that for $\sigma$-almost every $(x_1,x_2)\in X\times X$ we have $(a,x_1)\in\gen(\lambda_{(a,x_1)},\Psi)$.
By \cref{prop:sigma-s-full},  $\sigma$-almost every $(x_1,x_2)\in X\times X$ also has the property that $(x_1,x_2)\in\supp(\lambda_{(x_1,x_2)})$.
Using \cref{lemma_samelambda} it follows that $\sigma$-almost every $(x_1,x_2)\in X\times X$ satisfies $\lambda_{(x_1,x_2)}=\lambda_{(a,x_1)}$.
We conclude that $\sigma$-almost every $(x_1,x_2)$ satisfies both of the following properties (matching Properties 1 and 2 in \cref{sec:outline-of-proof}):
\begin{itemize}
\item
$(a,x_1)\in\gen(\lambda_{(a,x_1)},\Psi)$,
\item
$(x_1,x_2)\in\supp(\lambda_{(a,x_1)})$.

\end{itemize}
Since orbits of generic points are dense in the support (see, eg., \cite[Lemma 2.4]{KMRR}), we deduce that for $\sigma$-almost every $(x_1,x_2)\in X^2$, the point $(a,x_1,x_2)\in X^3$ is an \Erdos{} progression.

To finish the proof, note that if $(a,x_1,x_2)\in X^3$ is an \Erdos{} progression then $(a,x_1,x_2)$ satisfies the conclusion of \cref{thm_maindynamicalcontinuouskronecker}, part~\ref{itm_maindynamicalcontinuouskronecker_i}, if and only if
\[
(x_1,x_2)\in E\times T^{-t}E
\]
for some $t\in\N$, whereas
$(a,x_1,x_2)$ satisfies the conclusion of part~\ref{itm_maindynamicalcontinuouskronecker_ii} if and only if
\[
(x_1,x_2)\in (T\times T)^{-t}(E\times E)
\]
for some $t\in\N$.
Therefore, the proof is complete once we verify that
\begin{equation}
    \label{eq:positive}
\sigma\biggl(E\times \biggl(\bigcup_{t\in\N}T^{-t}E\biggr)\biggr)>0 { \qquad\text{and}\qquad \sigma\left(\bigcup_{t\in\N} (T\times T)^{-t} (E\times E)\right)>0.}
\end{equation}
Since $\mu$ is ergodic and $E$ has positive measure, the union  $\bigcup_{t\in\N}T^{-t}E$ covers all of $X$ up to a set of measure $0$ (with respect to $\mu$).
Writing
$$E\times \biggl(\bigcup_{t\in\N}T^{-t}E\biggr)
=
\big(E\times X\big)\cap \left(X\times \biggl(\bigcup_{t\in\N}T^{-t}E\biggr)\right),$$ we use \cref{sigma_pi2_covers_mu} and then \cref{lemma_projectionsofsigma} to obtain
\begin{align*}
\sigma\biggl(E\times \biggl(\bigcup_{t\in\N}T^{-t}E\biggr)\biggr)
=
\sigma(E\times X)
=
\mu(E)
>
0,
\end{align*}
as desired.

We are left with showing the second positivity statement in~\eqref{eq:positive}.
 Write $u=2t$ when $u$ is even and $u=2t+1$ when $u$ is odd.
Since $\sigma$ is $(T\times T^2)$-invariant, $\sigma\left((T\times T)^{-2t}(E\times E)\right)=\sigma(T^{-t}E\times E)$ and  $\sigma\left((T\times T)^{-2t-1}(E\times E)\right)=\sigma(T^{-t-1}E\times T^{-1}E)$.
On the other hand,
$$\sigma\left(\bigcup_{t\in\N}T^{-t}E\times E\ \cup\ \bigcup_{t\in\N}T^{-t-1}E\times T^{-1}E\right)=\sigma\big(X\times (E\cup T^{-1}E)\big)>0.$$
    Therefore, for some $u\in\N$ we have that $\sigma\left((T\times T)^{-u}(E\times E)\right)>0$.
\end{proof}

%AFFILIATION
\bigskip
\footnotesize
\noindent
Bryna Kra \\
\textsc{Northwestern University}\par\nopagebreak
\noindent
\href{mailto:kra@math.northwestern.edu}
{\texttt{kra@math.northwestern.edu}}

\bigskip
\footnotesize
\noindent
Joel Moreira\\
\textsc{University of Warwick} \par\nopagebreak
\noindent
\href{mailto:joel.moreira@warwick.ac.uk}
{\texttt{joel.moreira@warwick.ac.uk}}

\bigskip
\footnotesize
\noindent
Florian K.\ Richter\\
\textsc{{\'E}cole Polytechnique F{\'e}d{\'e}rale de Lausanne} (EPFL)\par\nopagebreak
\noindent
\href{mailto:f.richter@epfl.ch}
{\texttt{f.richter@epfl.ch}}

\bigskip
\footnotesize
\noindent
Donald Robertson\\
\textsc{University of Manchester} \par\nopagebreak
\noindent
\href{mailto:donald.robertson@manchester.ac.uk}
{\texttt{donald.robertson@manchester.ac.uk}}

\end{document}